\documentclass[12pt,a4paper]{amsart}
 \usepackage{amssymb,amsthm}
  \usepackage{multirow}
  \usepackage{amsmath}
\usepackage{amsfonts}
\usepackage{hyperref}
\usepackage{graphicx}
\usepackage{amsmath,cite}
\usepackage{bookmark}
 \newtheorem{thm}{Theorem}[section]
\newtheorem{cor}[thm]{Corollary}
\newtheorem{prop}[thm]{Proposition}
\newtheorem{lem}[thm]{Lemma}

\theoremstyle{definition}

\begin{document}
 \title[Parity of an odd dominating set]
 {Parity of an odd dominating set}
 \author{Ahmet Batal}
 \address{Department of Mathematics\\ Izmir Institute of Technology\\Izmir, TURKEY}
\email{ahmetbatal@iyte.edu.tr}
 \date{}
\begin{abstract}
For a simple graph $G$ with vertex set $V(G)=\{v_1,...,v_n\}$, we define the closed neighborhood set of a vertex $u$ as $N[u]=\{v \in V(G) \; | \; v \; \text{is adjacent to} \; u \; \text{or} \; v=u \}$ and the closed neighborhood matrix $N(G)$ as the matrix obtained by setting to $1$ all the diagonal entries of the adjacency matrix of
$G$. We say a set $S$ is odd dominating if $N[u]\cap S$ is odd for all $u\in V(G)$. We prove that the parity of an odd dominating set of $G$ is equal to the parity of the rank of $G$, where the rank of $G$ is defined as the dimension of the column space of $N(G)$. Using this result we prove several corollaries in one of which we obtain a general formula for the nullity of the join of graphs.
\end{abstract}

\subjclass[2020]{05C69}
\keywords{Lights Out, all-ones problem, odd dominating set, parity domination, domination number.}
 \maketitle

 \section{Introduction}

 Let $N[u]$ denote the \emph{closed neighborhood set} of a vertex $u$ in a simple graph $G$, i.e.; $$N[u]=\{v \in V(G) \; | v \; \text{is} \; \text{adjacent} \; \text{to} \; u \; \text{or} \; v=u \}.$$ Then, we say a subset $S$ of vertices is \emph{odd (even) dominating} if $N[u]\cap S$ is odd (even) for all $u\in V(G)$. In general, for an arbitrary subset $C$ of vertices, we say a set $S$ is a $C$-\emph{parity} set if $N[u]\cap S$ is odd for all $u\in C$ and even otherwise \cite{Amin96}. If there is a $C$-parity set for a given set $C$, we say that $C$ is \emph{solvable}. If there exists a $C$-parity set for every set $C$ of vertices in a graph $G$, then we say $G$ is \emph{always solvable}.

 %The odd and even dominating sets have close relation with the Lights Out game, which can be played on any simple graph. Initially every vertex of the graph has a configuration state \emph{on} or \emph{off}. The aim of the game is to turn all vertices \emph{off} by activating vertices, where each activation of a vertex $u$ switches the states of $u$ and every vertex adjacent to $u$. One can easily realize that activating a vertex twice is equivalent not to activate it. Also order of activations has no importance. Therefore an activation sequence can be represented by the set of activated vertices. Also an initial \emph{on}/\emph{off} configuration can be represented by the set of \emph{on} vertices. A set of activated vertices $S$, which turns all vertices \emph{off} for a given initial configuration $C$ is called a solving pattern for configuration $C$. In such a case we also say $C$ is solvable and $S$ solves $C$.

 Let $n$ be the order of $G$, $V(G)=\{v_1,...,v_n\}$ and $W$ be a subset of $V(G)$. The column vector $\mathbf{x}_W=(x_1,...,x_n)^t$, which is defined as $x_i=1$ if $v_i \in W $ and $x_i=0$ otherwise, is called the \emph{characteristic vector} of $W$. The closed neighbourhood matrix $N=N(G)$ of a graph
$G$ is obtained by setting to $1$ all the diagonal entries of the adjacency matrix of
$G$. Equivalently, $N(G)$ is the matrix whose $i$th column is equal to $\mathbf{x}_{N[v_i]}$. It is easy to observe that $S$ is a $C$-parity set if and only if
 \begin{equation}
 \label{NGsc}
 N(G)\mathbf{x}_S=\mathbf{x}_C
 \end{equation}
 over the field $\mathbb{Z}_2$ \cite{Sutner89},\cite{Sutner90}.

Let us denote the vectors whose components are all $0$ and all $1$ by $\mathbf{0}$ and $\mathbf{1}$, respectively. Then the following are equivalent. (\emph{a1}) $S$ is an odd dominating set, (\emph{a2}) $S$ is a $V(G)$-parity set, (\emph{a3}) $N(G)\mathbf{x}_S=\mathbf{1}$. Similarly, (\emph{b1}) $S$ is an even dominating set, (\emph{b2}) $S$ is a $\emptyset$-parity set, (\emph{b3}) $N(G)\mathbf{x}_S=\mathbf{0}$, are equivalent statements. Note that every graph has an even dominating set, which is $\emptyset$. On the other hand, it is proved by Sutner that every graph has an odd dominating set as well \cite{Sutner89} (see also \cite{Caro96}, \cite{Cowen99}, \cite{Erikson04}).

 %Let us denote the vectors whose components are all $0$ and all $1$ by $\mathbf{0}$ and $\mathbf{1}$, respectively. Let us also define the dot product $\mathbf{x}\cdot \mathbf{y}$ of two vectors $\mathbf{x}$ and $\mathbf{y}$ over the field $\mathbb{Z}_2$ as  $\mathbf{x}\cdot \mathbf{y}:= \mathbf{x}^t \mathbf{y} $. Then, it is easy to observe that the followings are equivalent. (\emph{a1}) $S$ is an odd dominating set, (\emph{a2}) $S$ is a solving pattern for $V(G)$, (\emph{a3}) $N(G)\mathbf{x}_S=\mathbf{1}$, (\emph{a4})  $\mathbf{x}_{N[u]}\cdot \mathbf{x}_S =1$ for all $u\in V(G)$. Similarly, (\emph{b1}) $S$ is an even dominating set, (\emph{b2}) $S$ is a solving pattern for $\emptyset$, (\emph{b3}) $N(G)\mathbf{x}_S=\mathbf{0}$, (\emph{b4})  $\mathbf{x}_{N[u]}\cdot \mathbf{x}_S =0$ for all $u\in V(G),$ are equivalent statements. Note that every graph has an even dominating set, which is $\emptyset$. On the other hand, it was first proved by Sutner, REFERENCE, that every graph has an odd dominating set as well.

Let $Ker(N)$ and $Col(N)$ denote the  kernel and column space of $N$, respectively. Let $\nu(G):=dim(Ker(N(G))$ and $\rho(G):=dim(Col(N(G))$. We call $\nu(G)$, the \emph{nullity} of $G$ (Amin et al. \cite{Amin98} call it the parity dimension of $G$) and $\rho(G)$, the \emph{rank} of $G$. We have $\nu(G)+\rho(G)=n$ by the rank nullity theorem.

From the matrix equation \eqref{NGsc}, we see that $G$ is always solvable if and only if $\nu(G)=0$. Moreover, $\nu(G)>0$ if and only if $G$ has a nonempty even dominating set.

We write $pr(a)$ to denote the parity function of a number $a$, i.e.; $pr(a)=0$ if $a$ is even and $pr(a)=1$ if $a$ is odd. In the case where $A$ is a matrix, $pr(A)$ is the parity function of the sum of its entries. For a set $S$, we write $pr(S)$ to denote the parity function of the cardinality of $S$ and say the \emph{parity of} $S$ instead of the parity of the cardinality of $S$. Note that $pr(S)= pr( \mathbf{x}_S )$. It was first noticed by Amin et al. [\cite{Amin92}, Lemma 3], and follows immediately from Sutner’s theorem, that for a given graph, the parity of all odd dominating sets are the same. Hence, the value of $pr(S)$, where $S$ is an odd dominating set of a graph is independent of the particular odd dominating set $S$ taken into account.

Our main result Theorem \ref{main} states that the parity of an odd dominating set is equal to the parity of the rank of the graph.
%Let $\overline{\mathbf{x}}:= \mathbf{x}+\mathbf{1}$ be the inverse of a vector $\mathbf{x}$. Note that $\overline{\mathbf{x}_S}=\mathbf{x}_{S^c}$.
 \section{Main result}

\begin{lem}
\label{prA}
Let $A$ be a $n\times n$, symmetric, invertible matrix over the field $\mathbb{Z}_2$ with diagonal entries equal to $1$. Then
$ pr(A^{-1}) = pr(A) = pr(n).$

\end{lem}

\begin{proof}
In the proof, all algebraic operations are considered over the field $\mathbb{Z}_2$. First of all, note that since $A$ is a symmetric matrix with nonzero diagonal entries, we have
\begin{equation}
 pr(A)=\sum_{i,j}A_{ij}= \sum_{i}A_{ii}= \sum_{i}1=pr(n). \nonumber
 \end{equation}
Similarly,
\begin{equation}
 pr(A^{-1})= \sum_{i}(A^{-1})_{ii}. \nonumber
 \end{equation}
On the other hand,
%We define the matrices $B$ and $\tilde{B}$ with entries $B_{ij}=\tilde{B}_{ij}=1$ if $A_{ij}=(A^{-1})_{ij}=0$, and $B_{ij}=A_{ij}$, $\tilde{B}_{ij}=(A^{-1})_{ij}$ if otherwise. Then
%\begin{equation}
%\label{AB}
%\sum_{i,j}A_{ij}=\sum_{i,j} B_{ij},\;\;\;    \sum_{i,j}(A^{-1})_{ij} =\sum_{i,j} \tilde{B}_{ij},\;\;\; \sum_{i,j} A_{ij}(A^{-1})_{ij}=\sum_{i,j} B_{ij}\tilde{B}_{ij}
%\end{equation}
%because summands differ only if both $A_{ij}$ and $(A^{-1})_{ij}$ are zero, and this may happen only an even number of times since $A$ and $A^{-1}$ are symmetric and diagonal entries of $A$ are not zero.
%On the other hand, we have $B_{ij}\tilde{B}_{ij}=B_{ij}+\tilde{B}_{ij}+1$ since $B_{ij}$ and $\tilde{B}_{ij}$ are not zero at the same time. Consequently,

\begin{align}
  pr(n)=Tr(I) & = Tr(AA^{-1}) \nonumber \\
              & = \sum_{i,j} A_{ij}(A^{-1})_{ij} \nonumber \\ \nonumber
              & = \sum_{i} A_{ii}(A^{-1})_{ii}\\ \nonumber
              & = \sum_{i} (A^{-1})_{ii}.\\ \nonumber
\end{align}
\end{proof}

We call a vertex a \emph{null vertex} of a graph $G$ if it belongs to an even dominating set of $G$. Since the set
of all characteristic vectors for even dominating sets of $G$ is a subspace of the vector space of all
binary $n$-tuples, if $v$ is a null vertex of G, then precisely half of the even dominating sets of $G$
contain $v$.

\begin{lem}
\label{j1} Let $G$ be a graph and  $v$ be a null vertex of $G$. Then there exists an odd dominating set of $G$ which does not contain $v$.
\end{lem}

\begin{proof}
Let $R$ be an even dominating set containing $v$ and $S_1$ be an odd dominating set of $G$. Assume $S_1$ contains $v$, otherwise we are done. Let $S_2$ be the symmetric difference of $S_1$ and $R$. Clearly $S_2$ is an odd dominating set which does not contain $v$.
\end{proof}

Let $G-v$ denote the graph obtained by removing a vertex $v$ and all its incident edges from a graph $G$. The number $nd(v):=\nu(G-v)-\nu(G)$ is called the \emph{null difference number}. It turns out that $nd(v)$ can be either $-1$, $0$, or $1$. Moreover, Ballard et al. proved the following lemma in [\cite{Ballard19}, Proposition 2.4.].

\begin{lem}[\cite{Ballard19}]
\label{j2}
Let $v$ be a vertex of a graph $G$. Then $v$ is a null vertex if and only if  $nd(v)=-1$.
\end{lem}

% Let $G_1$ and $G_2$ be two nonempty disjoint graphs with $u\in V(G_1)$ and $w\in V(G_2)$.   We form the join graph $H:=G_1uwG_2$ by connecting the vertices $u$ of $G_1$ and $w$ of $G_2$ by an edge. In [\cite{Batal20}, Theorem 3.3], we stated what the nullity of $H$ and the activation numbers of $u$ and $w$ with respect to $H$ become when we join two graphs with specific activation numbers of $u$ and $w$ with respect to $G_1$ and $G_2$, respectively. Moreover, we explained the structure of the component sets of any odd dominating set of $H$ in graphs $G_1$ and $G_2$. To prove our main result Theorem \ref{main}, we need a particular case from [\cite{Batal20}, Theorem 3.3], which can be stated as follows.
% \begin{thm}[\cite{Batal20}]
% \label{thm1}
%Let $G_1$ and $G_2$ be two nonempty disjoint graphs with $u\in V(G_1)$ and $w\in V(G_2)$, and let $H:=G_1uwG_2$. If $\mathcal{A}_{G_1}(u)=-1$ and $\mathcal{A}_{G_2}(w)=1$, then $\mathcal{A}_{H}(u)=1$ and $\mathcal{A}_{H}(w)=0$ and $\nu(H)=\nu(G_1)+\nu(G_2)-1$. Moreover, every odd dominating set $S$ of $H$ has the form $S=S_1\cup S_2$ such that $S_1\subseteq V(G_1)$, $S_2\subseteq V(G_2)$, and $S_1$ is an odd dominating set of $G_1$ and $S_2$ is $V(G_2)\backslash \{w\}$-parity set of $G_2$.
% \end{thm}

Now we are ready to state our main result.
\begin{thm}
\label{main}
Let $G$ be a graph and  $S$ be an odd dominating set of $G$. Then $pr(S)= pr(\rho(G))$. Equivalently, $pr(V(G)\backslash S)= pr(\nu(G))$.
\end{thm}
\begin{proof}
We prove the claim by applying induction on the nullity of the graph. Let $n$ be the order of $G$. In the case where $\nu(G)=0$, there exists a unique odd dominating set $S$ such that $N \mathbf{x}_S=\mathbf{1}$. Note that $N$ satisfies the conditions of Lemma \ref{prA}. Hence, together with the rank nullity theorem, we have $$pr(S)=pr(\mathbf{x}_S)= pr( N^{-1} \mathbf{1} )= pr(N^{-1})=pr(N)= pr(n)= pr (\rho(G)).$$

Now assume that $\nu(G)>0$ and the claim holds true for all graphs with nullity less than  $\nu(G)$. Since $\nu(G)$ is nonzero, there exists a non-empty even dominating set. Hence, there exists a null vertex $v$ of $G$. By Lemma \ref{j1}, there is an odd dominating set $S$ of $G$ which does not contain $v$. Since $S$ does not contain $v$, it is also an odd dominating set of the graph $G-v$. Moreover, by Lemma \ref{j2}, $nd(v)=-1$. Hence, $\nu(G-v)=\nu(G)+nd(v)=\nu(G)-1<\nu(G)$. By the induction hypothesis $pr(S)=pr (\rho(G-v))$. On the other hand, using the rank nullity theorem we obtain $\rho(G-v) = n-1-\nu(G-v)= n-1-\nu(G)+1 = n-\nu(G)= \rho(G).$ We complete the proof by noting that all odd dominating sets in $G$  have the same parity.
\end{proof}
\section{Some corollaries}

%\begin{lem}
%\label{lemma1}
%Let $G$ be a graph and  $S$ be an odd dominating set of $G$. If $| V(G)\backslash S |$ is odd, then there exists $u \in V(G)\backslash S $ with odd degree.
%\end{lem}
%\begin{proof}
%Consider the subgraph $K$ induced by $V(G)\backslash S$. Since $| V(G)\backslash S |$ is odd, by the handshake lemma, there is a vertex $u \in V(G)\backslash S $ such that $\deg_K(u)$ is even. Moreover, $deg_G(u)= %\deg_K(u)+|N[u]\cap S|$ and $|N[u]\cap S|$ is odd since $S$ is an odd dominating set. Hence, $\deg_G(u)$ is odd.
%\end{proof}
%Together with Theorem \ref{main}, the above lemma implies the following.
\begin{cor}
Let $G$ be an always solvable graph of order $n$. Then the odd dominating set of $G$ has odd (even) cardinality if $n$ is odd (even).
\end{cor}

Note that if every vertex of a graph $G$ has even degree, then $V(G)$ itself is an odd dominating set. This, together with Theorem \ref{main}, gives the following.
\begin{cor}
If every vertex of a graph $G$ has even degree, then $\nu(G)$ is even.
\end{cor}
\begin{cor}
If the number of even degree vertices of a tree $T$ is at most one, then every odd dominating set of $T$ has odd cardinality.
\end{cor}
\begin{proof}
Let $n$ be the order of $T$. By [\cite{Amin98}, Theorem 3] if every vertex of $T$ has odd degree, then $\nu(T)=1$. By the handshaking lemma, $n$ must be even, hence
$\rho(T)$ is odd. By [\cite{Amin98}, Theorem 4], if exactly one vertex of $T$ has even degree, then $\nu(T)=0$.
Since $n$ must be odd, $\rho(T)$ is also odd. Hence in either case, every odd dominating set has
odd cardinality by Theorem \ref{main}.
\end{proof}
\begin{cor}
Every odd dominating set of a graph $G$  has an odd (even) number of vertices of odd degree if and only if $\nu(G)$ is odd (even). In particular, the odd dominating set of an always solvable graph has an even number of odd degree vertices.
\end{cor}
\begin{proof}
Observe that for any subsets $A$, $B$ of $V(G)$, $pr(A\cap B)=\mathbf{x}_A^t \mathbf{x}_B$. In particular, $pr(A)=\mathbf{x}_A^t \mathbf{1}$. Let $A^c$ be the complement of $A$ in $V(G)$. Then we have $\mathbf{x}_{A^c}=\mathbf{x}_A+\mathbf{1}$. Now let $S$ be an odd dominating set of $G$ and $D$ be the set of vertices with odd degree. Observe that $N\mathbf{1}=\mathbf{x}_{D^c}$. Therefore $N \mathbf{x}_{S^c}= N (\mathbf{x}_S+\mathbf{1})=\mathbf{1}+\mathbf{x}_{D^c}=\mathbf{x}_D$. Then,
$pr(D\cap S)= \mathbf{x}_D^t \mathbf{x}_S = (N\mathbf{x}_{S^c})^t \mathbf{x}_S =\mathbf{x}_{S^c}^t N\mathbf{x}_S =\mathbf{x}_{S^c}^t \mathbf{1}=pr(S^c).$
On the other hand, $pr(S^c)=pr(\nu(G))$ by Theorem \ref{main}. Hence, the result follows.
\end{proof}

We define the \emph{join} $G_1\oplus ... \oplus G_m$ of $m$ pairwise disjoint graphs $G_1,...,G_m$ as follows. We take the vertex set as $V(G_1\oplus ... \oplus G_m)=\cup_{i=1}^m V(G_i)$ and the edge set as $E(G_1\oplus ... \oplus G_m)=\cup_{i=1}^m E(G_i)\cup \{(u,v)\;|\; u\in V(G_k) ,\; v\in V(G_l) \; k,l\in \{1,...,m\}\; \text{such that}\; k\neq l\}$. Then Amin et al. prove the following proposition in [\cite{Amin02}, Corollary 6].
\begin{prop}[\cite{Amin02}]
$\nu(G_1\oplus G_2)=\nu(G_1)+\nu(G_2)$ if either $G_1$ or $G_2$ has an odd dominating set of even cardinality, and $\nu(G_1\oplus G_2)=\nu(G_1)+\nu(G_2)+1$, otherwise.
\end{prop}
Together with Theorem \ref{main}, the above proposition implies the following.
\begin{equation}
\label{pro1}
\nu(G_1\oplus G_2)=\nu(G_1)+\nu(G_2)+pr(\rho(G_1)\rho(G_2)).
\end{equation}
Equivalently,
\begin{equation}
\label{pro2}
\rho(G_1\oplus G_2)=\rho(G_1)+\rho(G_2)-pr(\rho(G_1)\rho(G_2)).
\end{equation}
Equivalence of \eqref{pro1} and \eqref{pro2} follows from the rank nullity theorem.

Expressing the nullity/rank of $G_1\oplus G_2$ as a single formula involving nullities/ranks of $G_1$ and $G_2$ as above enables us to extend this result and to write a formula for the nullity/rank of the join of arbitrary number of graphs as follows.

\begin{prop}
Let $\{G_1,...,G_m\}$ be a collection of pairwise disjoint graphs. Let $j$ be the number of graphs in $\{G_1,...,G_m\}$ with odd rank. Then
\begin{align}
\label{rho1}
\nu(G_1\oplus ... \oplus G_m)= \left\{ \begin{array}{cc}
                \sum_{i=1}^m \nu(G_i)&  \;\;\;\;\text{if}\;\;\;\; j=0\\
                \;\; \;\; \;\;\;\;\;\;\;\;\sum_{i=1}^m \nu(G_i)+ j-1 &  \;\;\text{otherwise}  \\
                \end{array} \right\}.
\end{align}
Equivalently,
\begin{align}
\label{rho2}
\rho(G_1\oplus ... \oplus G_m)= \left\{ \begin{array}{cc}
                \sum_{i=1}^m \rho(G_i)&  \;\;\;\;\text{if}\;\;\;\; j=0\\
                \;\; \;\; \;\;\;\;\;\;\;\;\sum_{i=1}^m \rho(G_i)- j+1 &  \;\;\text{otherwise}  \\
                \end{array} \right\}.
\end{align}
\end{prop}
\begin{proof}
We prove \eqref{rho2}, then \eqref{rho1} follows from the rank nullity theorem. If $j=0$, then all graphs have even rank and the result follows applying \eqref{pro2} successively. Now let $j\neq 0$. Without loss of generality, we can assume that the first $j$ graphs have odd rank. Then, by \eqref{pro2}, $\rho(G_1 \oplus G_2) = \rho(G_1) + \rho(G_2)–1$, which is odd. Hence,
$\rho(G_1 \oplus G_2 \oplus G_3) = \rho(G_1) + \rho(G_2) -1 + \rho(G_3) – 1 = \rho(G_1) + \rho(G_2) + \rho(G_3) – 2$, which is
odd, and so on, yielding $\rho (G_1 \oplus G_2 \oplus \cdots \oplus G_j) = \rho(G_1) + \rho(G_2) + \cdots + \rho(G_j) – (j-1)$, which is
odd. Since the rank of the joins of the $m-j$ even ones is the sum of the ranks (which is even), the join of all $m$ of them is the sum of the
ranks minus $(j-1)$.

\end{proof}

%We also want to mention that Lemma \ref{lemma1}, together with the fact that every even dominating set has even cardinality REFERENCE, gives alternative proofs to the Corollary 7, and Corollary 8 in CARO2001 REFERENCE.

\textbf{Declaration of Competing Interests} The author declares that he has no
known competing financial interests or personal relationships that could have appeared to influence the work reported in this paper. \medskip

\textbf{Acknowledgements} The author would like to thank the referees for their
valuable suggestions which improved the clarity and quality of the paper.

\bibliographystyle{plain}

\begin{thebibliography}{10}

\bibitem{Amin92}
Amin, A.~T., Slater, P.~J.,
\newblock Neighborhood domination with parity restrictions in graphs.
\newblock In {\em Proceedings of the {T}wenty-third {S}outheastern
  {I}nternational {C}onference on {C}ombinatorics, {G}raph {T}heory, and
  {C}omputing ({B}oca {R}aton, {FL}, 1992)}, 91 (1992), 19--30.

\bibitem{Amin96}
Amin, A.~T., Slater, P.~J.,
\newblock All parity realizable trees,
\newblock {\em J. Combin. Math. Combin. Comput.}, 20 (1996), 53--63.

\bibitem{Amin98}
Amin, A.~T., Clark, L.~H., Slater, P.~J.,
\newblock Parity dimension for graphs,
\newblock {\em Discrete Math.}, 187(1-3) (1998), 1--17. https://doi.org/10.1016/S0012-365X(97)00242-2

\bibitem{Amin02}
Amin, A.~T., Slater, P.~J., Zhang, G.~H.,
\newblock Parity dimension for graphs---a linear algebraic approach,
\newblock {\em Linear Multilinear Algebra}, 50(4) (2002), 327--342. https://doi.org/10.1080/0308108021000049293

\bibitem{Ballard19}
Ballard, L.~E., Budge, E.~L., Stephenson, D.~R.,
\newblock Lights out for graphs related to one another by constructions,
\newblock {\em Involve}, 12(2) (2019), 181--201. https://doi.org/10.2140/involve.2019.12.181

\bibitem{Caro96}
Caro, Y.,
\newblock Simple proofs to three parity theorems,
\newblock {\em Ars Combin.}, 42 (1996), 175--180.

\bibitem{Cowen99}
Cowen, R., Hechler, S.~H., Kennedy, J.~W., Ryba, A.,
\newblock Inversion and neighborhood inversion in graphs,
\newblock {\em Graph Theory Notes N. Y.}, 37 (1999), 37--41.


\bibitem{Erikson04}
 Eriksson, H., Eriksson, K., Sj\"{o}strand, J.,
\newblock Note on the lamp lighting problem,
\newblock \emph{Special issue in honor of Dominique Foata's 65th birthday
  (Philadelphia, PA, 2000)}, 27 (2001), 357--366. https://doi.org/10.1006/aama.2001.0739

\bibitem{Sutner89}
Sutner, K.,
\newblock Linear cellular automata and the {G}arden-of-{E}den,
\newblock {\em Math. Intelligencer}, 11(2) (1989), 49--53. https://doi.org/10.1007/BF03023823

\bibitem{Sutner90}
Sutner, K.,
\newblock The {$\sigma$}-game and cellular automata,
\newblock {\em Amer. Math. Monthly}, 97(1) (1990), 24--34. https://doi.org/10.1080/00029890.1990.11995540

\end{thebibliography}

\label{son}

\end{document}